\documentclass[a4paper,11pt]{amsart}
\usepackage{amsmath,amssymb,amsthm}
\usepackage{amscd}
\usepackage{array,enumerate}
\newtheorem{Thm}{Theorem}[section]
\newtheorem{Prop}[Thm]{Proposition}
\newtheorem{Lem}[Thm]{Lemma}

\theoremstyle{definition}
\newtheorem{Rem}[Thm]{Remark}

\newtheorem{Def}[Thm]{Definition}
\newtheorem{Exm}[Thm]{Example}

\newcommand{\Cs}{\mbox{${\rm C}^\ast$}}
\newcommand{\id}{\mbox{\rm id}}

\newcommand{\coker}{\mbox{\rm coker}}

\title[Minimal skew products]{Construction of minimal skew products of amenable minimal dynamical systems}
\author{Yuhei Suzuki}
\subjclass[2000]{ Primary 37B05 ; Secondary 54H20}
\keywords{\Cs -algebras; amenable actions; pure infiniteness}
\address{Department of Mathematical Sciences,
University of Tokyo, Komaba, Tokyo, 153-8914, Japan}
\email{suzukiyu@ms.u-tokyo.ac.jp}

\begin{document}

\begin{abstract}
For an amenable minimal topologically free dynamical system $\alpha$ of a group on a compact metrizable space $Z$ and for a compact metrizable space $Y$
satisfying a mild condition, we construct a minimal skew product extension of $\alpha$ on $Z\times Y$.
This generalizes a result of Glasner and Weiss.
We also study the pure infiniteness of the crossed products of minimal dynamical systems arising from this result.
In particular, we give a generalization of a result of R\o rdam and Sierakowski.
\end{abstract} 
\maketitle
\section{Introduction}\label{Sec:intro}
Recall that a topological dynamical system $\Gamma \curvearrowright X$ is said to be minimal
if every $\Gamma$-orbit is dense in $X$.
It is an interesting question to ask that for a given group $\Gamma$ which space admits a minimal (topologically) free dynamical system of $\Gamma$.
Certainly a space admitting a minimal $\Gamma$-dynamical system must have a nice homogeneity.
However, this is not sufficient even for the simplest case, that is, the case $\Gamma=\mathbb{Z}$.
For example, an obstruction from homological algebra shows that
there is no minimal homeomorphism on even dimensional spheres $S^{2n}$ (see Chapter I.6 of \cite{Br} for instance).

In \cite{GW}, Glasner and Weiss have shown the existence of minimal skew product extensions of
a minimal homeomorphism under mild conditions.
Their result in particular shows that many spaces admit a minimal homeomorphism.
For example, it follows that there exists a minimal homeomorphism on the product of the Hilbert cube and $S^1$.
This solved a question asked by T. Chapman \cite{Ch}.
For certain amenable groups, their result is generalized in \cite{Ne}.
(It also deals generalizations of other results in \cite{GW}; e.g., the existence of strictly ergodic skew products.)
In this paper, following the argument of Glasner and Weiss in \cite{GW},
we construct minimal skew products of amenable minimal topologically free dynamical systems (Theorem \ref{Thm:min}).
This provides many new examples of (amenable) minimal topologically free dynamical systems of exact groups.

We also study the reduced crossed product of these minimal skew products.
Recall that a unital \Cs -algebra $A$ is purely infinite and simple if
for any nonzero positive element $a\in A$, there is $b\in A$ with $b^{\ast} ab=1$.
Pure infiniteness plays an important role in the study of \Cs -algebras.
See \cite{Cun}, \cite{Kir}, \cite{KP}, \cite{Phi}, and \cite{Rord} for example.
A \Cs -algebra is said to be a Kirchberg algebra if it is simple, separable, nuclear, and purely infinite.
A deep theorem of Kirchberg \cite{Kir} and Phillips \cite{Phi} states that the
Kirchberg algebras are classified in terms of the KK-theory.
In particular, the Kirchberg algebras in the UCT class are classified by their K-theoretic data,
and consequently each of which is isomorphic to the one constructed in \cite{Rord}.
For these reasons, it is important to know whether a given \Cs -algebra is purely infinite.
Obviously pure infiniteness implies other infiniteness properties; e.g., tracelessness, properly infiniteness.
The latter conditions are easy to check in many situations.
However, even in the nuclear case, R\o rdam has constructed a counterexample for the converse implications \cite{Rord2}.
See \cite{Ror} and the references therein for more information on pure infiniteness and Kirchberg algebras.
In Section \ref{Sec:pi}, under certain assumptions on $Y$ and $\alpha\colon \Gamma \curvearrowright Z$, we show that the crossed products of many of dynamical systems obtained in our result are
Kirchberg algebras in the UCT class (Proposition \ref{Prop:filling}).
For this purpose, we generalize the notion of the finite filling property,
which is introduced in \cite{JR} for dynamical systems,
to \'etale groupoids.
It turns out that the generalized version is useful to construct minimal skew products
with the purely infinite crossed products.
This result is applied particularly to the case that $Y$ is a connected closed topological manifold
and that $\alpha$ is a dynamical system on the Cantor set constructed in \cite{RS}.
As a consequence, we generalize a result of R\o rdam and Sierakowski \cite{RS}, which is a result for the Cantor set, to the products of connected closed topological manifolds and the Cantor set (Theorem \ref{Thm:RS}).
This is the first generalization of their result,
and shows that for topological dynamical systems, not only the structure of groups but also the structure of spaces
is not an obstruction to form a Kirchberg algebra.

In Section \ref{Sec:free}, we study the K-theory of the crossed products of these minimal skew products in
the free group case.
Using the Pimsner--Voiculescu six-term exact sequence,
we prove a K$\ddot{{\rm u}}$nneth-type formula for them.
As an application, for any connected closed topological manifold $M$
and for any (non-amenable, countable) virtually free group $\Gamma$,
we show that there exist continuously many amenable minimal free dynamical systems of $\Gamma$
on the product of $M$ and the Cantor set whose crossed products are mutually non-isomorphic Kirchberg algebras. This generalizes a result in \cite{Suz}.

\subsection*{Spaces of dynamical systems}
For a compact metrizable space $X$, let $\mathcal{H}(X)$ denote the group of homeomorphisms on $X$.
We equip the metric $d$ on $\mathcal{H}(X)$ as follows.
First let us fix a metric $d_X$ on $X$.
Then define
$$d(\varphi, \psi):=\max_{x\in X} (d_X(\varphi(x), \psi(x)))+\max_{x\in X}(d_X(\varphi^{-1}(x), \psi^{-1}(x)))$$
for $\varphi, \psi \in \mathcal{H}(X)$.
It is not hard to check that the metric $d$ is complete and $\mathcal{H}(X)$ becomes a topological group with respect to $d$.
Note that the sequence $(\varphi_n)_n$ in $\mathcal{H}(X)$ converges to $\varphi$ in this topology
if and only if $\varphi_n$ uniformly converges to $\varphi$.
For a countable group $\Gamma$, let
$\mathcal{H}(\Gamma, X)$ denote the set of dynamical systems of $\Gamma$ on $X$,
i.e., $\mathcal{H}(\Gamma, X)={\rm Hom}(\Gamma, \mathcal{H}(X))$.
This set is naturally regarded as a closed subset of $\prod_{\Gamma}\mathcal{H}(X)$.
Since $\Gamma$ is countable, this makes $\mathcal{H}(\Gamma, X)$ to be a complete metric space.

Next let $Y$ be a compact metrizable space and let $\mathcal{G} \curvearrowright Y$ be a continuous action of a topological group $\mathcal{G}$ on $Y$.
Let $\alpha\colon \Gamma \curvearrowright Z$ be a topological dynamical system
of a group $\Gamma$ on a compact metrizable space $Z$.
Put $X=Z\times Y$.
Recall that a continuous map $c\colon \Gamma \times Z  \rightarrow \mathcal{G}$ is said to be a cocycle if
it satisfies the equation
$c(s, t.z)c(t, z)=c(st, z)$ for all $s, t\in \Gamma$ and $z\in Z$.
When there is a continuous map $h\colon Z \rightarrow \mathcal{G}$
satisfying $c(s, z)=h(s.z)^{-1}h(z)$ for all $s\in \Gamma$ and $z\in Z$,
the cocycle $c$ is said to be a coboundary.
Each cocycle $c\colon \Gamma\times Z \rightarrow \mathcal{G}$
defines an extension of $\alpha$ on $X$ by the following equation.
$$s.(z, y)=(s.z, c(s, z)y) {\rm \ for\ } s\in \Gamma {\rm\ and\ } (z, y)\in X.$$
Such an extension is called a skew product extension.
Note that when $c$ is a coboundary, the associated skew product extension is conjugate
to $\bar{\alpha}$.
Here and throughout the paper,
for a dynamical system $\alpha \colon \Gamma \curvearrowright Z$
and a compact space $Y$,
we denote by $\bar{\alpha}$ the diagonal action of $\alpha$ and the trivial action on $Y$.
Since the space $Y$ is always clear from the context, we omit $Y$ in our notation.

For a continuous map $h$ from $Z$ into $\mathcal{G}$,
we have an associated homeomorphism $H$ on $X$
defined by the formula
$H(z, y):=(z, h_z(y))$ for $(z, y)\in X$.
We denote by $\mathcal{G}_s$ the set of homeomorphisms given in the above way.
Obviously, $\mathcal{G}_s$ is a subgroup of $\mathcal{H}(X)$.
For a topological dynamical system $\alpha\colon \Gamma \curvearrowright Z$,
we define a subset $\mathcal{S}_{\mathcal{G}}(\alpha)$ of $\mathcal{H}(\Gamma, X)$ to be
$$\mathcal{S}_{\mathcal{G}}(\alpha):=\{ H^{-1}\circ \bar{\alpha}\circ H: H\in \mathcal{G}_s\}.$$
We note that the set $\mathcal{S}_{\mathcal{G}}(\alpha)$ consists of skew product extensions of $\alpha$
by coboundaries.
We denote by $\overline{\mathcal{S}}_{\mathcal{G}}(\alpha)$ the closure of $\mathcal{S}_{\mathcal{G}}(\alpha)$
in $\mathcal{H}(\Gamma, X)$.
Note that any $\beta \in \overline{\mathcal{S}}_{\mathcal{G}}(\alpha)$ is a skew product extension of $\alpha$ on $X$
whose associated cocycle takes the value in $\overline{\mathcal{G}}$.
Here $\overline{\mathcal{G}}$ denotes the closure of the image of $\mathcal{G}$ in $\mathcal{H}(X)$.
In particular, when $\alpha$ is amenable,
every dynamical system contained in $\overline{\mathcal{S}}_{\mathcal{G}}(\alpha)$ is amenable.
Throughout the paper, we always fix metrics $d_Y$ and $d_Z$ on $Y$ and $Z$ respectively and
consider the metric on $X=Z\times Y$
defined by $d_X((z_1, y_1), (z_2, y_2))= d_Y(y_1, y_2) + d_Z(z_1, z_2)$,
and use these metrics to define metrics on the homeomorphism groups.

In Section \ref{Sec:pi}, we discuss \'etale groupoids.
Throughout the paper, we always assume that \'etale groupoids are locally compact Hausdorff and
their unit spaces are compact and infinite (as a set).
For an \'etale groupoid $G$, we denote by $r$ and $s$ the range and source map
unless they are specified.
As usual, for a dynamical system $\alpha$ of a discrete group $\Gamma$ on a compact space $X$,
we usually regard the transformation groupoid $X\rtimes _\alpha \Gamma$
as the following subspace of $X \times \Gamma \times X$.
$$X\rtimes _\alpha \Gamma = \{ (\alpha_g(x), g, x)\in X \times \Gamma \times X: x\in X, g\in \Gamma\}.$$
Note that the range and source map correspond to
the projections onto the first and third coordinate respectively.
For detailed explanations and basic knowledges of \'etale groupoids, we refer the reader to Section 5.6 of \cite{BO}.

\subsection*{Notation}
\begin{itemize}
\item
For a subset $U$ of a topological space,
its closure and interior are denoted by ${\rm cl}(U)$ and ${\rm int}(U)$ respectively.
\item
For a $\ast$-homomorphism $\alpha$ between \Cs -algebras,
denote by $\alpha_{\ast, i}$
the homomorphism induced on the $K_i$-groups.
\item Denote by $\mathbb{K}$ the \Cs -algebra of all compact operators on $\ell^2(\mathbb{N})$.
\item Let $A$ be a \Cs -algebra.
For a projection $p$ in $A$ or $A\otimes \mathbb{K}$,
denote by $[p]_0$ the element of $K_0(A)$ represented by $p$.
\item
For a compact space $X$,
we denote $K_i(C(X))$ by $K^i(X)$ for short.
(Note that this coincides with the usual definition of $K^i$-group.)

\end{itemize}
\section{Construction of minimal skew product}\label{Sec:min}
The goal of this section is to prove the following theorem.
The proof is done by following the same line as that of Theorem 1 in \cite{GW}.

Before the proof, recall that a dynamical system $\alpha\colon \Gamma \curvearrowright Z$ of a group $\Gamma$ on a compact metrizable space $Z$ is said to be amenable if
there is a sequence of continuous maps
$$\mu_n\colon Z \rightarrow {\rm Prob}(\Gamma)$$
satisfying $$\lim_{n\rightarrow \infty} \sup_{z\in Z}\| s.\mu_n^z- \mu_n^{s.z}\|_1 =0 {\rm\ for\ all\ }s\in \Gamma.$$
Here ${\rm Prob}(\Gamma)$ denotes the space of probability measures on $\Gamma$ with the pointwise convergence topology,
and $\Gamma$ acts on ${\rm Prob}(\Gamma)$ by the left translation.
It is shown by Ozawa \cite{Oz} that for discrete groups,
the existence of an amenable action is equivalent to exactness.
See \cite{Ana0} and \cite{BO} for more information on amenable actions.
In the proof of the following theorem, we use amenability of dynamical systems
to construct suitable continuous functions.
In other word, amenability of dynamical systems plays the role of the F\o lner sets in the proof of Theorem 1 of \cite{GW}.
\begin{Thm}\label{Thm:min}
Let $\mathcal{G}\curvearrowright Y$ be a minimal action of a path connected group $\mathcal{G}$ on a compact metrizable space $Y$.
Let $\alpha\colon \Gamma \curvearrowright Z$ be an amenable minimal topologically free dynamical system of a countable group $\Gamma$ on a compact metrizable space $Z$.
Then the set
$$\{\beta \in \overline{\mathcal{S}}_{\mathcal{G}}(\alpha): \beta{\rm\ is\ minimal}\}$$
is a $G_\delta$-dense subset of $\overline{\mathcal{S}}_{\mathcal{G}}(\alpha)$.
\end{Thm}
\begin{proof}
Let $\mathcal{G}\curvearrowright Y$ and $\alpha\colon \Gamma\curvearrowright Z$ be as in the statement.
For an open set $U$ of $X=Z\times Y$,
we define the subset $\mathcal{E}_U$ of $\overline{\mathcal{S}}_{\mathcal{G}}(\alpha)$ to be
$$\mathcal{E}_U:=\{\beta\in \overline{\mathcal{S}}_{\mathcal{G}}(\alpha):\bigcup_{g\in \Gamma}\beta_g(U)=X \}.$$
Since $X$ is compact, it is not hard to check that
the set $\mathcal{E}_U$ is open in $\overline{\mathcal{S}}_{\mathcal{G}}(\alpha)$.

Let $(U_n)_n$ be a countable basis of $X$.
We observe that
an element in $\overline{\mathcal{S}}_{\mathcal{G}}(\alpha)$ is minimal
if and only if it is contained in $\bigcap_n \mathcal{E}_{U_n}$.
Therefore, thanks to the Baire category theorem, our claim follows once we show the density of $\mathcal{E}_U$ in $\overline{\mathcal{S}}_{\mathcal{G}}(\alpha)$ for each non-empty open set $U$ in $X$. 
To see this, it is enough to show the following claim.
For any $H\in \mathcal{G}_s$ and any non-empty open set $U\subset X$,
$H^{-1}\circ\bar{\alpha}\circ H \in {\rm cl}(\mathcal{E}_U).$
This is equivalent to the condition
$\bar{\alpha}\in {\rm cl}(H\mathcal{E}_UH^{-1})$.
A direct computation shows that
$H\mathcal{E}_UH^{-1}=\mathcal{E}_{H(U)}$.
Since $H(U)$ is again a non-empty open set,
now it is enough to show the following statement.
For any non-empty open set $U\subset X$, we have $\bar{\alpha}\in {\rm cl}(\mathcal{E}_U)$.
Now let $U$ be a non-empty open set in $X$.
Let $S$ be a finite subset of $\Gamma$ and let $\epsilon>0$.
Take non-empty open sets $V\subset Y$ and $W\subset Z$ with
$W\times V\subset U$.
By assumption, there are $\tilde{h}_0, \ldots, \tilde{h}_n\in \mathcal{G}$
satisfying $\bigcup_{0\leq i\leq n} \tilde{h}_i(V)=Y$.
Since $\mathcal{G}$ is path-connected,
there is a continuous map $h \colon [0, 1]\rightarrow \mathcal{G}$
satisfying $h_{i/n}=\tilde{h}_i$ for $0\leq i\leq n$.
By the continuity of $h$,
there is $\delta>0$ such that the condition
$|t_1-t_2|<\delta$ implies $d(h_{t_1}^{-1}h_{t_2}, \id_Y)<\epsilon$.
Now we use the amenability of $\alpha$ to choose a continuous map $\mu\colon Z\rightarrow {\rm Prob}(\Gamma)$ satisfying
$\sup_{z\in Z}\|s.\mu^z-\mu^{s.z}\| _1<\delta$ for all $s\in S$.
By perturbing $\mu$ within a small error and replacing $W$ by a smaller one,
we may assume that there is a finite set $F\subset \Gamma$
such that ${\rm supp}(\mu^w)\subset F$ for all $w\in W$.
(Cf.\ Lemma 4.3.8 of \cite{BO}.)
Since $\alpha$ is topologically free,
by replacing $W$ by a smaller one further, we may assume that the open sets $(g.W)_{g\in F^{-1}}$ are mutually disjoint.
Since $W$ is a locally compact metrizable space without isolated points,
we can choose a compact subset $K$ of $W$ homeomorphic to the Cantor set.
(To see this, take a sequence of families $((K_{i_1, \ldots, i_n})_{0\leq i_1, \ldots, i_n \leq 1})_{n\in \mathbb{N}}$
satisfying the following conditions.
Each family $(K_{i_1, \ldots, i_n})_{i_1, \ldots, i_n}$
consists of pairwise disjoint closed subsets of $W$ with non-empty interior,
$K_{i_1, \ldots, i_n}\subset K_{i_1, \ldots, i_{n-1}}$ for any $i_1, \ldots, i_n$,
and $\max_{i_1, \ldots, i_n}\{ {\rm diam}(K_{i_1, \ldots, i_n})\}$ converges to $0$ as $n$ tends to infinity.
Then the set $\bigcap_n \bigcup _{i_1, \ldots, i_n} K_{i_1, \ldots, i_n}$ gives the desired subset.)

Next take a continuous surjection
$\theta_0\colon K\rightarrow [0, 1]$.
Extend $\theta_0$ to a map $\bigsqcup_{g\in F^{-1}} g K\rightarrow [0, 1]$
by the formula
$\theta_0(g.z):=\theta_0(z)$ for $g\in F^{-1}$ and $z\in K$.
Then take a continuous extension
$\tilde{\theta}\colon Z\rightarrow [0, 1]$ of $\theta_0$.
Using $\tilde{\theta}$ and $\mu$, we define
$\theta\colon Z\rightarrow [0, 1]$
by
$$\theta(z):=\sum_{g\in \Gamma}\mu^z(g^{-1})\tilde{\theta}(g.z).$$
Note that the continuity of $\tilde{\theta}$ and $\mu$ implies
that of $\theta$.
For $z\in K$, since ${\rm supp}(\mu^z)\subset F$,
we have $\theta(z)=\theta_0(z)$.
In particular, $\theta(K)=[0, 1]$.
Moreover, for $z\in Z$ and $s\in S$,
we have
\begin{eqnarray*}
|\theta(s.z)-\theta(z)|&=& |\sum_{g\in \Gamma}(\mu^{s.z}(g^{-1})\tilde{\theta}(gs.z)-\mu^{z}(g^{-1})\tilde{\theta}(g.z))|\\
&=&|\sum_{g\in \Gamma}(\mu^{s.z}(g^{-1})\tilde{\theta}(gs.z)-\mu^{z}(s^{-1}g^{-1})\tilde{\theta}(gs.z))|\\
&\leq&\|\mu^{s.z}-s.\mu^z\|_1\\
&<&\delta.
\end{eqnarray*}

Now define the map $g\colon Z\rightarrow \mathcal{G}$ by
$g_z:=h_{\theta(z)}$ for $z\in Z$.
We will show that the corresponding homeomorphism $G\in \mathcal{G}_s$ satisfies the following conditions.
\begin{enumerate}
\item
$d(\bar{\alpha}_s, G^{-1}\circ \bar{\alpha}_s \circ G)<\epsilon$ for $s\in S$.
\item
$G^{-1}\circ \bar{\alpha} \circ G\in \mathcal{E}_U$.
\end{enumerate}
Since $U$, $\epsilon$, and $S$ are arbitrarily, this ends the proof.
Let $s\in S$ and $(z, y)\in X$.
Then a direct computation shows that
$$(G^{-1}\circ \bar{\alpha}_s \circ G)(z, y)
=(\alpha_s(z), g_{s.z}^{-1} g_z(y)).$$
Since $d(g_{s.z}^{-1} g_z, \id_Y)< \epsilon$ for all $z \in Z$,
we obtain the first condition.

For the second condition,
note that $G^{-1}\circ \bar{\alpha} \circ G\in \mathcal{E}_U$
if and only if
$\bigcup_{g\in \Gamma} \bar{\alpha}_g (G(U))=X$ holds.
By the choice of $G$,
for any $0\leq i \leq n$,
there is $w\in W$ satisfying $g_w= \tilde{h}_{i}$.
It follows that for any $0\leq i \leq n$, there is $w\in W$
with $\{w \}\times \tilde{h}_i(V)\subset G(U)$.
Since $\bigcup_i \tilde{h}_i(V)=Y$, this shows that
for any $y\in Y$,
the intersection $(Z\times \{y\}) \cap G(U)$ is non-empty (which is open in $Z\times \{y \}$).
This with the minimality of $\alpha$ shows that
$\bigcup_{g\in \Gamma}\bar{\alpha}_g(G(U))=X$. 
\end{proof}
\begin{Rem}
Theorem \ref{Thm:min} does not hold when $Z$ is not metrizable.
To see this, consider a minimal subsystem $\alpha\colon \Gamma \curvearrowright Z$ of $\Gamma \curvearrowright \beta \Gamma$. (Note that $\alpha$ is amenable when $\Gamma$ is exact.)
Then $\alpha$ is the universal minimal $\Gamma$-system (see Theorem 1.24 of \cite{Gla}).
Thus it does not have a nontrivial minimal extension.
\end{Rem}
\begin{Rem}
Let $\alpha\colon \Gamma \curvearrowright Z$ be a minimal topologically free dynamical system of an amenable group $\Gamma$
whose crossed product is quasi-diagonal.
Then for any $\beta \in \overline{\mathcal{S}}_{\mathcal{G}}(\alpha)$,
its crossed product is quasi-diagonal.
Indeed, since $\beta$ is a limit of conjugations of $\bar{\alpha}$,
there is a continuous field of \Cs -algebras over
$\mathbb{N}\cup \{\infty \}$, the one-point compactification of $\mathbb{N}$,
whose fiber at $n\in \mathbb{N}$ is isomorphic to $C(Y)\otimes (C(Z)\rtimes_\alpha \Gamma)$
and the one at $\infty$ is isomorphic to $C(X)\rtimes _\beta \Gamma$.
(See Corollary 3.6 of \cite{Rie}.)
Now Lemma 3.10 of \cite{CDE} proves the quasi-diagonality of $C(X)\rtimes _\beta \Gamma$.
\end{Rem}

\section{Pure infiniteness of crossed products of minimal skew products}\label{Sec:pi}
\subsection{Finite filling property for \'etale groupoids}
To study the pure infiniteness of crossed products of dynamical systems arising from Theorem \ref{Thm:min},
we introduce a notion of the finite filling property for \'etale groupoids.
First recall from \cite{JR} the finite filling property for dynamical systems.
Although their definition and result also cover noncommutative \Cs -dynamical systems,
in this paper, we concentrate on the commutative case.
See \cite{JR} for the general case.
We remark that, although the following formulation is slightly different from the original one,
it is easily checked that they are equivalent.
(Cf.\ Definition 0.1, Proposition 0.3, and Remark 0.4 of \cite{JR}.)
\begin{Def}
A dynamical system $\Gamma\curvearrowright X$ is said to have the $n$-filling property
if for any non-empty open set $U$ of $X$,
there are $n$ elements $g_1, \ldots, g_n \in \Gamma$
with $\bigcup_{i=1}^n g_i(U)=X$.
We say that a dynamical system has the finite filling property if it has the $n$-filling property for some $n\in \mathbb{N}$.
\end{Def}
Note that the finite filling property implies minimality.
In \cite{JR}, it is shown that the finite filling property of a topological dynamical system
implies the pure infiniteness of the reduced crossed product
by a similar way to the one in \cite{LaS}.
However, as shown in \cite{JR}, the $n$-filling property is inherited to factors.
This makes the usage of the $n$-filling property restrictive in our application.
To avoid this difficulty, we introduce a notion of the finite filling property
for \'etale groupoids, which can be regarded as a localized version of \cite{JR}.
This helps to construct minimal skew products with the purely infinite reduced crossed product.

Next we recall that a subset $U$ of an \'etale groupoid $G$ is said to be a $G$-set
if both the range and source map are injective on $U$.
For two $G$-sets $U$ and $V$, we set
$UV:=\{uv\in G: u\in U, v\in V, s(u)= r(v)\}$.
Obviously it is again a $G$-set.
Furthermore, if both $U$ and $V$ are open, then $UV$ is again open.
Recall also that an \'etale groupoid is said to be minimal
if for any $x\in G^{(0)}$,
the set $\{r(u):u\in G, s(u)=x\}$ is dense in $G^{(0)}$.
Note that the unit space $G^{(0)}$ has no isolated points whenever $G$ is minimal.
(Recall that $G^{(0)}$ is always assumed to be infinite.)
\begin{Def}
Let $G$ be an \'etale groupoid.
For a natural number $n$, we say that $G$ has the $n$-filling property if every non-empty open set $W$ of $G^{(0)}$ satisfies the following conditon.
There are $n$ open $G$-sets $U_1, \ldots, U_n$ satisfying
$$\bigcup_{i=1}^n r(U_i W)=G^{(0)}.$$
For short, we say that a dynamical system has the weak $n$-filling (resp.\ weak finite filling) property if its transformation groupoid has the $n$-filling (resp.\ finite filling) property. 
\end{Def}

Obviously, for dynamical systems, the $n$-filling (resp.\ finite filling) property implies the weak $n$-filling (resp.\ weak finite filling) property.
However, the converses are not true.

We also remark that it is possible to define the weak finite filling property without going through the transformation groupoid.
However, this specialization does not make the arguments below easier
and this generality makes notation simpler.
Considering applications elsewhere also, we study the property under this generality.

When the unit space $G^{(0)}$ has finite covering dimension,
we have a useful criteria for the finite filling property.
The following definition is inspired from \cite{Mat} and \cite{RS}.
\begin{Def}
We say that an \'etale groupoid $G$ is purely infinite if
for any non-empty open set $U$ of $G^{(0)}$,
there is a non-empty open subset $V$ of $U$ with the following condition.
There are open $G$-sets $U_1$ and $U_2$
such that $r(U_i) \subset V\subset s(U_i)$ for $i= 1, 2$
and $r(U_1)$ and $r(U_2)$ are disjoint.
We say that a dynamical system is purely infinite if its transformation groupoid is purely infinite.
\end{Def}
We remark that Matui \cite{Mat} has introduced pure infiniteness for totally disconnected \'etale groupoids
for the study of the topological full groups.
Clearly, our definition is weaker than Matui's one.
We will see later that our definition of pure infiniteness coincides with Matui's one
for minimal totally disconnected \'etale groupoids.
\begin{Prop}\label{Prop:pifill}
Let $G$ be a minimal purely infinite \'etale groupoid
and assume that $\dim(G^{(0)})=n<\infty$.
Then $G$ has the $(n+1)$-filling property.
\end{Prop}
\begin{proof}
Let $U$ be a non-empty open subset of $G^{(0)}$.
Replacing $U$ by a smaller one, we may assume
that there are open $G$-sets $U_1$ and $U_2$ such that
$r(U_i)\subset U\subset s(U_i)$ for $i=1, 2$ and $r(U_1)$ and $r(U_2)$ are disjoint.
We first show that for any $N\in \mathbb{N}$, there are $N$ open $G$-sets $V_1, \ldots, V_N$
satisfying $r(V_i)\subset U\subset s(V_i)$ for $i=1, \ldots, N$ and the ranges $r(V_1), \ldots, r(V_N)$ are mutually disjoint.
To see this, first take $M\in \mathbb{N}$ with $2^M\geq N$ and then
take $N$ mutually distinct elements from the set $$\{U_{i_1} U_{i_2} \cdots U_{i_M}: i_k=1 {\rm\ or\ }2 {\rm\ for\ each\ }k\}.$$
Then it gives the desired sequence.

By the compactness of $G^{(0)}$ and the minimality of $G$, for some natural number $N$, there are $N$ open
$G$-sets $W_1, \ldots, W_N$ with
$\bigcup_{i=1}^N r(W_i U)=G^{(0)}.$
Take $N$ open $G$-sets $V_1, \ldots, V_N$ as in the previous paragraph
and put $Z_i:=W_i V_i^{-1}$ for each $i$.
Then we have 
$$\bigcup_{i=1}^N r(Z_i U)\supset \bigcup_{i=1}^N r(W_i U)=G^{(0)}.$$
Note that since $s(Z_i)\subset r(V_i)$, the sources of $Z_i$'s are mutually disjoint.
Since $\dim (G^{(0)})=n$,
we can choose a refinement $(Y_j)_{j\in J}$ of $(r(Z_i U))_{i=1} ^N$
with the decomposition $J=J_0\sqcup J_1 \sqcup \cdots \sqcup J_n$
such that the members of the family $(Y_j)_{j\in J_k}$ are mutually disjoint for each $k$.
Choose a map $\varphi\colon J \rightarrow \{1, \ldots, N \}$
satisfying $Y_j \subset r(Z_{\varphi(j)} U)$ for each $j\in J$.
Set
$X_k:= \bigcup_{j\in J_k} Y_j Z_{\varphi(j)}$ for each $k$.
Then it is not hard to check that each $X_k$ is an open $G$-set and that
$r(X_k U)=\bigcup_{j\in J_k}Y_j$.
This shows 
$\bigcup_{k=0}^{n}r(X_k U)=G^{(0)}$.
\end{proof}
\begin{Rem}\label{Rem:pi}
The argument in Remark 4.12 of \cite{Mat} shows that
for totally disconnected \'etale groupoids,
the finite filling property implies pure infiniteness in Matui's sense.
Thus for a minimal totally disconnected \'etale groupoid $G$, pure infiniteness in Matui's sense \cite{Mat},
that in our sense, the finite filling property, and the $1$-filling property are equivalent.
(Here total disconnectedness is used to replace open $G$-sets by clopen ones.)
\end{Rem}
Next we see a few examples of dynamical systems with the weak finite filling property.
The following three examples are particularly important for us.
See \cite{JR} for more examples of dynamical systems with
the finite filling property.
\begin{Exm}\label{Exm:1fill}
It follows from the proof of Theorem 6.11 of \cite{RS} that every countable non-amenable exact group
admits an amenable minimal free purely infinite dynamical system on the Cantor set.
(To see this, use the equivalence of conditions (i) and (iii) in Proposition 5.5 in the proof of Proposition 6.8.)
By Proposition \ref{Prop:pifill}, it has the weak $1$-filling property.
We remark that these dynamical systems almost never have the finite filling property.
\end{Exm}
Recall that a manifold is said to be closed if it is compact and has no boundaries.
\begin{Lem}\label{Lem:man}
Let $M$ be a connected closed topological manifold.
Let $\mathcal{H}_0(M)$ denote the path connected component of $\mathcal{H}(M)$ containing the identity.
Then the action $\mathcal{H}_0(M)\curvearrowright M$ has the finite filling property.
\end{Lem}
\begin{proof}
It is not hard to show that
the above action is transitive by using the connectedness of $M$
with the fact that $M$ is locally homeomorphic to $\mathbb{R}^n$.

Take an open cover $U_1, \ldots, U_N$ of $M$ each of which is homeomorphic to $\mathbb{R}^n$.
We show that for any non-empty open set $V$ in $M$, for any $i$, and for any compact subset $K$ of $U_i$,
there is an element $g\in \mathcal{H}_0(M)$ with
$g(V)\supset K$.
Since $M$ is compact, the claim with a standard argument for compactness shows the $N$-filling property
of the action in the question.
Since the action is transitive, replacing $V$ by $g(V)$ for a suitable $g\in \mathcal{H}_0(X)$ and replacing it by a smaller one further,
we may assume that $V$ is contained in $U_i$.
Take a homeomorphism
$\varphi\colon U_i\rightarrow \mathbb{R}^n$ satisfying
$0 \in \varphi(V)$.
Take a sufficiently large positive number $\lambda >0$
with $\varphi(K)\subset \lambda\varphi(V)$.
Then choose a continuous function $f\colon \mathbb{R}_{\geq 0} \rightarrow \mathbb{R}_{\geq 0}$
satisfying the following conditions.
\begin{enumerate}
\item
For $t \leq {\rm diam}(\varphi(V))$, we have $f(t)=\lambda$.
\item
For all sufficiently large $t$, we have $f(t)= 1$.
\item
The function $t\mapsto tf(t)$ is strictly monotone increasing.
\end{enumerate}
Now set $\varphi_f(x):=\varphi^{-1}(f(\| \varphi(x)\|) \varphi(x))$ for $x\in U_i$.
Here $\|\cdot \|$ denotes the Euclidean norm on $\mathbb{R}^n$.
From the assumptions on $f$,
the map $\varphi_f$ is a homeomorphism on $U_i$
satisfying $K \subset \varphi_f(V)$.
We extend $\varphi_f$ to a homeomorphism $\psi_f$ on $M$ as follows.
\[\psi_f(x):=\left\{ \begin{array}{ll}
\varphi_f(x) & {\rm if\ } x\in U_i \\
x &{\rm if\ } x\in M\setminus U_i\\
\end{array}.\right.\]
It is clear from the properties of $f$ that $\psi_f$ is indeed a homeomorphism on $M$.
Clearly we have $K \subset \psi_f(V)$.
Moreover, the map
$t\in [0, 1] \mapsto \psi_{(1-t)f+ tk}$ defines
a continuous path in $\mathcal{H}(M)$ from $\psi_f$ to the identity.
Here $k$ denotes the constant function of value $1$ defined on $\mathbb{R}_{\geq 0}$.
Thus we have $\psi_f \in \mathcal{H}_0(M)$.
\end{proof}
Next we see examples of finite filling actions of path-connected groups on infinite dimensional spaces.
Let $Q:=\prod_{\mathbb{N}} [0, 1]$ be the Hilbert cube.
Recall that a topological space is said to be a Hilbert cube manifold if
there is an open cover each of the member is homeomorphic to an open subset of $Q$.
It is not hard to show that open subsets of $Q$ in the definition can be taken to be $[0, 1)\times Q$.
(See Theorem 12.1 of \cite{Ch} for instance.)
Obvious examples are $Q$ itself and the product
of $Q$ and a topological manifold (possible with boundary).
We refer the reader to \cite{Ch} for more information of Hilbert cube manifolds.
\begin{Lem}\label{Lem:Qman}
Let $M$ be a connected compact Hilbert cube manifold.
Then the action $\mathcal{H}_0(M)\curvearrowright M$ has
the finite filling property.
\end{Lem}
\begin{proof}
We first show the following claim.
For any open subset $U$ of $[0, 1)\times [0, 1]^n$ of the form
$(a, b)^n\times [0, 1]$ ($0<a<b < 1$)
and for any compact subset $K$ of $[0, 1)\times [0, 1]^n$,
there is a homeomorphism
$h\in \mathcal{H}_{c, 0}([0, 1) \times [0, 1]^n)$ satisfying
$K \subset h(U)$.
Here, for a locally compact metrizable space $Y$, $\mathcal{H}_{c, 0}(Y)$ denotes the subgroup of homeomorphisms on $Y$ defined as follows.
First we define $\mathcal{H}_{c}(Y)$ to be the group of homeomorphisms on $Y$ which coincide with the identity off a compact subset.
Then we identify $\mathcal{H}_{c}(Y)$ with the inductive limit of subgroups of homeomorphism groups of compact subsets of $Y$ in the natural way.
Then we topologize $\mathcal{H}_{c}(Y)$ with the inductive topology.
Now we define $\mathcal{H}_{c, 0}(Y)$ to be the path-connected component of $\mathcal{H}_{c}(Y)$ containing the identity
with respect to this topology.
To show the claim, we first construct a homeomorphism $h_1 \in \mathcal{H}_{c, 0}([0, 1)\times [0, 1]^n)$
satisfying $h_1(\{0 \}\times [0, 1]^n) \subset (a, b)^n\times [0, 1]$ in a similar way to the proof of Lemma \ref{Lem:man}.
Then, since $h_1$ is a homeomorphism,
there is a positive number $\delta>0$ satisfying
$h_1([0, \delta)\times [0, 1]^n )\subset (a, b)^n\times [0, 1].$
On the one hand, it is easy to find $h_2 \in \mathcal{H}_{c, 0}([0, 1)\times [0, 1]^n)$
satisfying
$K \subset h_2([0, \delta)\times [0, 1]^n).$
Now the homeomorphism $h:=h_2\circ h_1^{-1}$ satisfies the required condition.

Next we observe that for any compact metrizable space $X$ and its open subset $U$,
any $h\in \mathcal{H}_{c, 0}(U)$ extends to a homeomorphism $\tilde{h}$ in $\mathcal{H}_0(X)$
by defining $\tilde{h}(x)=x$ off $U$.
Now thanks to the claim in the previous paragraph with this observation,
the rest of the proof can be completed by a similar way to that of Lemma \ref{Lem:man}.
\end{proof}

We next show that the finite filling property gives a sufficient condition for the pure infiniteness of the reduced groupoid \Cs -algebra.
Recall from \cite{Mat} that an \'etale groupoid $G$ is said to be essentially principal if
the interior of the set
$\{g\in G: r(g)=s(g)\}$ coincides with $G^{(0)}$.
Note that for transformation groupoids, this condition is equivalent to the topological freeness
of the original dynamical system.
\begin{Prop}\label{Prop:pi}
Let $G$ be an \'etale groupoid with the finite filling property.
Assume further that $G$ is essentially principal.
Then the reduced groupoid \Cs -algebra ${\rm C}^\ast_{\rm r}(G)$ is purely infinite and simple.
In particular, if $G$ is additionally assumed to be second countable and amenable,
then ${\rm C}^\ast_{\rm r}(G)$ is a Kirchberg algebra in the UCT class.
\end{Prop}
We note that the last statement immediately follows from the first one since
the reduced groupoid \Cs -algebra of an amenable \'etale groupoid is nuclear (see Theorem 5.6.18 of \cite{BO}) and is in the UCT class \cite{Tu}.
To show the main statement, we need the following lemma, which is the analogue of Lemma 1.5 of \cite{JR}.
\begin{Lem}\label{Lem:pi}
Let $G$ be an \'etale groupoid with the $n$-filling property.
Let $b$ be a positive element in $C(G^{(0)})$ with norm one.
Then for any $\epsilon>0$, there is $c\in {\rm C}^\ast_{\rm r}(G)$ such that
$\|c \|\leq \sqrt{n}$ and $c^\ast b c\geq 1-\epsilon$.
\end{Lem}
\begin{proof}
Set $U:=\{x\in G^{(0)}: b(x)>1-\epsilon\}$.
Take $n$ mutually disjoint non-empty open subsets $U_1, \ldots, U_n$ of $U$.
Since $G$ is minimal, there are $n$ open $G$-sets $V_1, \ldots, V_n$
with the property that the intersection $\bigcap_i r(V_i U_i)$ is non-empty.
Using the $n$-filling property of $G$ with this observation, we can find
$n$ open $G$-sets $W_1, \ldots, W_n$ satisfying
$$\bigcup_{i=1}^n r(W_i U_i)=G^{(0)}.$$
By replacing $W_i$ by $W_i U_i$,
we may assume $s(W_i)\subset U_i$.
Since $G$ is locally compact and $G^{(0)}$ is compact,
replacing each $W_i$ by a smaller one if necessary,
we may assume further that each $W_i$ is relatively compact in $G$.
Since $G$ is locally compact, for each $i$, it is not hard to find
an increasing net $(W_{i, \lambda})_{\lambda \in \Lambda}$ of open subsets of $W_i$
that satisfies the following conditions.
The closure of $W_{i, \lambda}$ in $G$ is contained in $W_i$ for each $\lambda$,
and the union $\bigcup_\lambda W_{i, \lambda}$ is equal to $W_i$.
Since the unit space $G^{(0)}$ is compact,
there is $\lambda \in \Lambda$ satisfying
$\bigcup_{i=1}^n r(W_{i, \lambda})= G^{(0)}$.
Now fix such $\lambda$ and put $Z_i:= {\rm cl}(W_{i, \lambda})$ for each $i$.
Then, by the choice of $W_{i, \lambda}$, the $Z_i$ is a compact $G$-set.
Moreover we have
$$ G^{(0)} = \bigcup_{i=1}^n r(W_{i, \lambda})\subset \bigcup_{i=1}^n r(Z_i).$$
Now for each $i$, take a continuous function $f_i \in C_c(G)$ satisfying the following conditions.
\begin{enumerate}
\item $0 \leq f_i \leq 1.$
\item ${\rm supp}(f_i)\subset W_i$.
\item $f_i\equiv 1$ on $Z_i$.
\end{enumerate}
(Since $Z_i$ and the closure of $W_i$ in $G$ are compact, such function exists.)
Since $W_i$ is a $G$-set,
these conditions imply that
$f_i \ast f_i^\ast \in C(G^{(0)})$ and that
$f_i \ast f_i^\ast \leq 1$.
Since the sets $s(W_1), \ldots, s(W_n)$ are mutually disjoint,
we have $f_i \ast f_j^\ast =0$ for two distinct $i$ and $j$.
Now put $c:=\sum_{i=1}^n f_i^\ast$.
The above observations show that
$c^\ast \ast c\in C(G^{(0)})$ and that $c^\ast \ast c \leq n$.
Thus $\|c\| \leq \sqrt{n}$.
Since the $G$-sets $W_1, \ldots, W_n$ have mutually disjoint sources,
we also get
$c^\ast \ast b \ast c\in C(G^{(0)})$.
Since $s(W_i)\subset U$ for each $i$ and $\bigcup_{i=1}^n r(Z_i)=G^{(0)}$, we further obtain
$c^\ast \ast b \ast c\geq 1-\epsilon.$
\end{proof}
\begin{proof}[Proof of Proposition \ref{Prop:pi}]
The rest of the proof is basically the same as that in \cite{JR}.
We first observe that since $G$ is essentially principal, it is not hard to show that
for any $b\in C_c(G)$ and $\epsilon>0$,
there is a positive element $y\in C(G^{(0)})$ with norm one satisfying
$yby=yE(b)y$ and $\|yby \|> \|E(b)\|-\epsilon$, where $E$ denotes the restriction map $C_c(G)\rightarrow C(G^{(0)})$.
Note that the map $E$ extends to a faithful conditional expectation on ${\rm C}^\ast_{\rm r}(G)$.
From this with Lemma \ref{Lem:pi}, for any positive element $b\in C_c(G)$ with $\| E(b)\|=1$,
there is an element $c\in C_c(G^{(0)})$ satisfying
$\|c\|\leq \sqrt{n}$ and $c^{\ast} b c\geq 1/2.$
Since the norm of $c$ is bounded by the fixed constant $\sqrt{n}$,
now a standard argument completes the proof.
\end{proof}

\subsection{Minimal skew products with purely infinite crossed products}
Now using the finite and weak finite filling property,
we construct minimal skew products
whose crossed products are purely infinite.

\begin{Prop}\label{Prop:filling}
Let $\alpha\colon \Gamma \curvearrowright Z$ be an amenable topologically free dynamical system with the weak $n$-filling property.
Let $\mathcal{G}\curvearrowright Y$ be a minimal dynamical system of a path connected group $\mathcal{G}$ with the $m$-filling property.
Then the set
$$\left\{ \beta\in \overline{\mathcal{S}}_{\mathcal{G}}(\alpha): \beta {\rm \ has\ the\ weak\ }(nm){\rm \mathchar`-filling\ property}\right\}$$
is a $G_\delta$-dense subset of $\overline{\mathcal{S}}_{\mathcal{G}}(\alpha).$
\end{Prop}
\begin{proof}
For an open set $U$ of $X=Z\times Y$,
let $\mathcal{F}_U$ denote the set of elements $\beta$ of $\overline{\mathcal{S}}_{\mathcal{G}}(\alpha)$ satisfying the following condition.
There are $nm$ open $G_ \beta$-sets $V_1, \ldots, V_{nm}$ with
$\bigcup_i r(V_iU)= X$.
Here $G_\beta$ denotes the transformation groupoid $X\rtimes_\beta \Gamma$ of $\beta$.
Then for a countable basis $(U_n)_n$ of $X$,
the set in the question coincides with the intersection $\bigcap_n \mathcal{F}_{U_n}$.
Hence it suffices to show that
each $\mathcal{F}_U$ is open and dense in $\overline{\mathcal{S}}_{\mathcal{G}}(\alpha)$.

We first show the openness of $\mathcal{F}_U$.
Let $\beta \in \mathcal{F}_U$.
Let $V_1, \ldots, V_{nm}$ be open $G_\beta$-sets as above.
Replacing $V_i$'s by smaller ones, we may assume that they are relatively compact in $G_\beta$
and that the sources $s(V_i)$ are contained in $U$.
Set $F:=\pi(\bigcup_i V_i)$, where $\pi\colon X\rtimes_\beta \Gamma \rightarrow \Gamma$
denotes the projection onto the second coordinate.
Since each $V_i$ is relatively compact in $G_\beta$, the set $F$ is a finite subset of $\Gamma$.
Now we apply the argument in the proof of Lemma \ref{Lem:pi} to $(V_i)_i$ to choose compact $G_\beta$-sets $W_1, \ldots, W_{nm}$ with the following properties.
The $W_i$ is contained in $V_i$ for each $i$
and the union $\bigcup_i r({\rm int}(W_i))$ is equal to $X$.
Now for a $G_\beta$-set $W$ and $g\in \Gamma$,
define the subset $W_g\subset X$ to be
$r(W\cap \pi^{-1}(\{g\})).$
Then, for each $i$, the sets $(W_{i, g})_{g\in F}$ are mutually disjoint compact sets in $X$.
Moreover, the union $\bigcup_{i, g}{\rm int}(W_{i, g})$ is equal to $X$.

For $W \subset X$ and $\delta >0$, we define the (open) subsets $\mathcal{N}_\delta(W)$ and $\mathcal{I}_\delta(W)$ of $X$ as follows.
$$\mathcal{N}_\delta(W) := \bigcup_{x\in W} B(x, \delta),$$
$$\mathcal{I}_\delta(W):=\{x\in X: {\rm there\ is\ }\eta>\delta {\rm \ with\ }B(x, \eta)\subset W\}.$$
Here for $x \in X$ and $\eta >0$, $B(x, \eta)$ denotes the open ball of center $x$ and radius $\eta$.
Then, from the properties of $W_i$'s and the compactness of $X$,
for a sufficiently small positive number $\delta >0$, the following conditions hold.
The sets $(\mathcal{N}_\delta(W_{i, g}))_g$ are mutually disjoint for each $i$
and the sets $(\mathcal{I}_\delta(W_{i, g}))_{i, g}$ cover $X$.
We fix such positive number $\delta$.
From the first condition, for any $\gamma \in \overline{\mathcal{S}}_{\mathcal{G}}(\alpha)$
satisfying $d(\gamma_s, \beta_s)< \delta$ for all $s\in F$,
each $W_i$ is a $G_\gamma$-set.
Here $W_i$ is regarded as a subset of $G_\gamma$
by identifying the transformation groupoids with the set $\Gamma \times X$ by ignoring the first coordinates.
Let $r_\beta$ and $r_\gamma$ denote the range map of $G_\beta$ and $G_\gamma$ respectively.
Then we have
$$\bigcup_i r_\gamma ({\rm int} (W_i))\supset \bigcup_i \mathcal{I}_\delta ( r_\beta ({\rm int} (W_i)))
= \bigcup_{i, g} \mathcal{I}_\delta(W_{i, g}) =X.$$
Therefore we have $\gamma \in \mathcal{F}_U$, which proves the openness of $\mathcal{F}_U$.

To show the density of $\mathcal{F}_U$, by the similar reason to that in the proof of Theorem \ref{Thm:min},
it suffices to show the following statement.
For any $\epsilon >0$ and any finite subset $S\subset \Gamma$,
there is a homeomorphism $H \in \mathcal{G}_s$ satisfying the following conditions.
\begin{enumerate}
\item
$d(\bar{\alpha}_s, H^{-1}\circ \bar{\alpha}_s \circ H)<\epsilon$ for $s\in S$.
\item
$H^{-1}\circ \bar{\alpha}\circ H\in \mathcal{F}_U$.
\end{enumerate}
Replacing $U$ by a smaller open set, we may assume $U=W\times V$
for some $W\subset Z$ and $V\subset Y$.
By the $m$-filling property of $\mathcal{G}\curvearrowright Y$,
we can choose $m$ elements $\tilde{h}_1, \ldots, \tilde{h}_m$ of $\mathcal{G}$ with
$\bigcup_i \tilde{h}_i(V)= Y$.
Now proceeding the same argument as in the proof of Theorem \ref{Thm:min},
we get a continuous map
$g\colon Z\rightarrow \mathcal{G}$
with the following conditions.
\begin{enumerate}
\item
$d(g_{s.z}^{-1}g_z, \id_Y)<\epsilon$ for all $z\in Z$ and $s\in S$.
\item
There are $m$ elements $w_1, \ldots, w_m$ in $W$ with the condition
$\bigcup_i {g_{w_i}}(V)=Y$.
\end{enumerate}
Let $H \in \mathcal{G}_s$ be the element corresponding to $g$.
Then from the first condition, we conclude
$d(\bar{\alpha}_s, H^{-1}\circ \bar{\alpha}_s \circ H)<\epsilon$ for $s\in S$.
To show $\beta := H^{-1}\circ \bar{\alpha} \circ H\in \mathcal{F}_U$, it suffices to show the following claim.
There are $nm$ open $G_{\bar{\alpha}}$-sets $W_1, \ldots, W_{nm}$ with
$\bigcup_i r(W_i H(U))=X.$
Indeed the sets
$$\{(H^{-1}(z), s, H^{-1}(w))\in X \times \Gamma \times X: (z, s, w)\in W_i\}\ (i=1, \ldots, nm)$$
then define the desired open $G_\beta$-sets.
To show the claim, first note that since $g$ is continuous,
there are an open subset $U_i$ of $U$ containing $w_i$ for $i=1, \ldots, m$
and an open covering $(V_i)_{i=1}^m$ of $Y$ satisfying
the following condition.
For any $z\in U_i$, we have $V_i \subset g_z(V)$.
From these conditions, we have
$H(U)\supset \bigcup_{i=1}^m(U_i \times V_i)$.
Now for each $1\leq i \leq m$,
take $n$ open $G_\alpha$-sets $W_{i, 1}, \ldots, W_{i, n}$ with
$\bigcup_{j=1}^n r(W_{i, j} U_i)=Z$.
For each $1\leq i \leq m$ and $1\leq j \leq n$, set $Z_{i, j}:=\varphi^{-1}(W_{i, j})$,
where $\varphi\colon G_{\bar{\alpha}}\rightarrow G_\alpha$ denotes the canonical quotient map.
Then each $Z_{i, j}$ is an open $G_{\bar{\alpha}}$-set
and we further get
$$\bigcup_{i, j} r_{\bar{\alpha}} (Z_{i, j} H(U))\supset \bigcup_{i, j} r_{\bar{\alpha}} (Z_{i, j} (U_i \times V_i))=\bigcup_{i, j} (r_\alpha(W_{i, j}U_i)\times V_i) =X.$$
\end{proof}

In \cite{RS}, R\o rdam and Sierakowski have shown
that every countable non-amenable exact group admits an amenable minimal free dynamical system
on the Cantor set whose crossed product is a Kirchberg algebra in the UCT class.
Proposition \ref{Prop:filling} particularly gives an extension of their result to more general spaces.
\begin{Thm}\label{Thm:RS}
Let $M$ be a connected closed topological manifold, a connected compact Hilbert cube manifold,
or a countable direct product of these manifolds.
Let $X$ be the Cantor set.
Then every countable non-amenable exact group admits
an amenable minimal free dynamical system on $M\times X$
whose crossed product is a Kirchberg algebra in the UCT class.
\end{Thm}
\begin{proof}
For the first two cases, the statement immediately follows from Example \ref{Exm:1fill}, Lemmas \ref{Lem:man} and \ref{Lem:Qman}, and Propositions \ref{Prop:pi} and \ref{Prop:filling}.

For the last case, let $M_1, M_2, \ldots$ be a sequence of spaces
each of which is either connected closed topological manifold or connected compact Hilbert cube manifold.
Set $N_n:= M_1 \times \cdots \times M_n \times X$ for each $n$.
We put $\alpha_0:=\alpha$ and $N_0:=X$ for convenience.
We inductively apply Proposition \ref{Prop:filling} to
$\alpha_n\colon \Gamma\curvearrowright N_n$
and $M_{n+1}$ to get a minimal skew product extension $\alpha_{n+1}\colon \Gamma\curvearrowright N_{n+1}$ of $\alpha_n$ with the weak finite filling property.
Then we get the projective system $(\alpha_n)_{n=1}^\infty$ of dynamical systems of $\Gamma$.
Since pure infiniteness of \Cs -algebras is preserved under taking increasing union (Prop 4.1.8 of \cite{Ror}),
the projective limit $\alpha_\infty:=\varprojlim \alpha_n$ possesses the desired properties.
\end{proof}
\section{Minimal dynamical systems of free groups on products of Cantor set and closed manifolds}\label{Sec:free}
In this section, we investigate the K-groups of the crossed products of minimal dynamical systems obtained in
Theorem \ref{Thm:min} for the free group case.
By using the Pimsner--Voiculescu six term exact sequence \cite{PV},
we give a K$\ddot{{\rm u}}$nneth-type formula for K-groups of their crossed products.
As an application, we give the following generalization of Theorems 4.10 and 4.22 of \cite{Suz}.

\begin{Thm}\label{Thm:man}
Let $\Gamma$ be a countable non-amenable virtually free group.
Let $M$ be either connected closed topological manifold or connected compact Hilbert cube manifold.
Then there exist continuously many amenable minimal free dynamical systems of $\Gamma$ on the product
of $M$ and the Cantor set whose crossed products
are mutually non-isomorphic Kirchberg algebras.
\end{Thm}
In the below, we regard abelian groups as $\mathbb{Z}$-modules.
We simply denote the tensor product `$\otimes_{\mathbb{Z}}$' by `$\otimes$' for short.
Recall that for two abelian groups $G, H$, the group ${\rm Tor}_1^{\mathbb{Z}}(G, H)$ is defined as follows.
First take a projective resolution of $G$.
$$\cdots \rightarrow P_2\rightarrow P_1 \rightarrow P_0\rightarrow G \rightarrow 0.$$
Then by tensoring $H$ with the above resolution,
we obtain a complex
$$\cdots \rightarrow P_2\otimes H \rightarrow P_1\otimes H \rightarrow P_0 \otimes H \rightarrow 0.$$
The group  ${\rm Tor}_1^{\mathbb{Z}}(G, H)$ is then defined as the first homology of the above complex.
Note that the definition does not depend on the choice of the projective resolution.
We remark that when we have a projective resolution of length one
$$0\rightarrow P_1 \rightarrow P_0\rightarrow G \rightarrow 0,$$
then ${\rm Tor}_1^{\mathbb{Z}}(G, H)$ is computed as the kernel
of the homomorphism $P_1\otimes H\rightarrow P_0\otimes H$.
See \cite{Br} for the detail.
\begin{Prop}\label{Prop:free}
Let $\alpha\colon \mathbb{F}_d \curvearrowright X$ be an amenable minimal topologically free dynamical system of the free group $\mathbb{F}_d$
on the Cantor set $X$.
Let $\mathcal{G}\curvearrowright Y$ be a minimal action of a path-connected group $\mathcal{G}$ on a compact metrizable space $Y$.
Let $\beta\in \overline{\mathcal{S}}_{\mathcal{G}}(\alpha)$.
Let $A$ and $B$ denote the crossed product of $\alpha$ and $\beta$ respectively.
Then for $i=0, 1$, we have the following short exact sequence.
$$0\rightarrow K_0(A)\otimes K^i(Y) \rightarrow K_i(B)\rightarrow
(K_1(A)\otimes K^{1-i}(Y))\oplus {\rm Tor}_1^\mathbb{Z}(K_0(A), K^{1-i}(Y))\rightarrow 0.$$
Moreover, the first map maps $[1_A]_0\otimes [1_Y]_0$ to $[1_B]_0$ when $i=0$.
\end{Prop}
\begin{proof}
Since $C(X)$ is an AF-algebra,
we have a canonical isomorphism
$$K^i(X \times Y)\rightarrow C(X, K^i(Y)) (\cong K^0(X)\otimes K^i(Y))$$ for $i=0, 1$.
Here $C(X, K^i(Y))$ denotes the group of continuous maps from
$X$ into $K^i(Y)$ and $K^i(Y)$ is regarded as a discrete group.
For $i=0$, the isomorphism is given by mapping
the element $[p]_0$ where $p$ is a projection in $\mathbb{K}\otimes C(X)\otimes C(Y)$
to the map $x\in X\mapsto [p(x, \cdot)]_0\in K^0(Y)$ and similarly for the case $i=1$.

From this isomorphism and the fact that $\mathcal{G}$ is path-connected,
for any $\gamma \in \mathcal{S}_{\mathcal{G}}(\alpha)$ and $g\in \mathbb{F}_d$,
we have $(\gamma_g)_{\ast, i}=(\alpha_g)_{\ast, 0} \otimes \id_{K^i(Y)}$ for $i=0, 1$.
Here we identify $K^i(X\times Y)$ with $K^0(X)\otimes K^i(Y)$
under the above isomorphism.
By continuity of the K-theory, the above equality
holds for all $\gamma\in \overline{\mathcal{S}}_{\mathcal{G}}(\alpha)$.
Now let $S$ be a free basis of $\mathbb{F}_d$.
Then by the Pimsner--Voiculescu six term exact sequence \cite{PV},
we have the following short exact sequence.
$$0\rightarrow \coker(\varphi\otimes \id_{K^i(Y)})\rightarrow K_i(B)\rightarrow \ker(\varphi\otimes \id_{K^{1-i}(Y)})\rightarrow 0.$$
Here $\varphi$ denotes the homomorphism
$$\varphi\colon K^0(X)^{\oplus S}\rightarrow K^0(X)$$
which maps $(f_s)_{s\in S}$ to $\sum_{s\in S}(f_s- (\alpha_s)_{\ast, 0}(f_s))$. 
Since $K^0(X)$ is a free abelian group,
the exact sequence $$0\rightarrow K_1(A)\rightarrow K^0(X)^{\oplus S}\rightarrow K^0(X)\rightarrow K_0(A)\rightarrow 0$$
obtained by the Pimsner--Voiculescu six-term exact sequence
is a free resolution of $K_0(A)$.
This also gives the free resolution
$$0\rightarrow {\rm im}(\varphi)\rightarrow K^0(X)\rightarrow K_0(A)\rightarrow 0$$
of $K_0(A)$. Here the first map is given by the inclusion map, say $\iota$.

Let $\psi\colon K^0(X)^{\oplus S}\rightarrow {\rm im}(\varphi)$
be the surjective homomorphism obtained by restricting the range of $\varphi$.
By tensoring $K^i(Y)$ with the second free resolution, we obtain
the following exact sequence.
$$0\rightarrow {\rm Tor}_1^\mathbb{Z}(K_0(A), K^i(Y))\rightarrow {\rm im}(\varphi) \otimes K^i(Y) \rightarrow K^0(X)\otimes K^i(Y)\rightarrow K_0(A)\otimes K^i(Y)\rightarrow 0.$$
This shows that
$$\ker(\iota\otimes \id_{K^i(Y)})\cong {\rm Tor}_1^\mathbb{Z}(K_0(A), K^i(Y)).$$
Since the second map surjects onto ${\rm im}(\varphi\otimes \id_{K^i(Y)})$,
we also obtain the isomorphism
$$\coker(\varphi\otimes \id_{K^i(Y)})\cong K_0(A)\otimes K^i(Y).$$
Since $\varphi=\iota\circ \psi$ and $\psi$ is surjective, we have the following exact sequence.
\begin{equation}\label{eq:1}
0\rightarrow \ker(\psi \otimes \id_{K^i(Y)}) \rightarrow \ker(\varphi\otimes \id_{K^i(Y)}) \rightarrow \ker(\iota\otimes \id_{K^i(Y)}) \rightarrow 0.
\end{equation}
Here the first map is the canonical inclusion and the second map
is the restriction of $\psi\otimes \id_{K^i(Y)}$.
Since ${\rm im}(\varphi)$ is free abelian,
there is a direct complement $K$ of $\ker(\varphi)$ in $K^0(X)^{\oplus S}$.
Note that the restriction of $\psi$ on $K$ is an isomorphism.
Hence we have the isomorphism
$$\ker(\psi\otimes \id_{K^i(Y)})=\ker(\psi)\otimes K^i(Y)\cong K_1(A)\otimes K^i(Y).$$
Again by the freeness of ${\rm im}(\varphi)$, we have a right inverse $\sigma$ of $\psi$.
Then the homomorphism $\sigma\otimes \id _{K^i(Y)}$ gives a splitting of the short exact sequence (\ref{eq:1}).
Combining these observations, we obtain the isomorphism
$$\ker(\varphi\otimes \id_{K^i(Y)})\cong (K_1(A)\otimes K^i(Y))\oplus {\rm Tor}_1^\mathbb{Z}(K_0(A), K^i(Y)).$$
Now the first exact sequence completes the proof.
\end{proof}
\begin{Rem}
Certainly, when $K^\ast(Y)$ has a good property, the short exact sequence in Proposition \ref{Prop:free} is spilitting.
For example, it holds true when $K^{1-i}(Y)$ is projective or one of $K_0(A)$ or $K^i(Y)$ is injective.
(Recall that $K_1(A)$ is always free abelian
and that the tensor product of an injective $\mathbb{Z}$-module with an arbitrary $\mathbb{Z}$-module is
again injective by Corollary 4.2 of Ch.III of \cite{Br}.)
However, we do not know whether it is splitting in general.
Recall that a splitting of the K$\ddot{{\rm u}}$nneth tensor product theorem is obtained by
replacing considered \Cs -algebras
by easier ones by using suitable elements of the KK-groups (see Remark 7.11 of \cite{RSc}).
However, in our setting, this argument does not work.
Such replacement does not respect
the relation among $C(X), C(Y), A, B,$ and $\mathbb{F}_d$.
\end{Rem}
\begin{proof}[Proof of Theorem \ref{Thm:man}]
We first prove the claim for free groups.
Theorem 5.3 of \cite{Suz2} shows that for any finite $d$, there is an amenable minimal topologically free dynamical system $\gamma$
of $\mathbb{F}_d$ on the Cantor set whose crossed product $A$ satisfies the following condition.
The unit $[1]_0\in K_0(A)$ generates a direct summand of $K_0(A)$ isomorphic to $\mathbb{Z}$.
Note that this property passes to unital \Cs -subalgebras of $A$.
Moreover, since $\gamma$ is found as a factor of the ideal boundary action,
its restriction to any finite index subgroup of $\mathbb{F}_d$ is minimal.
It is also not hard to show that
the restriction of $\gamma$ to any finite index subgroup of $\mathbb{F}_d$ is purely infinite
(cf.\ the proof of Lemma 4.8 of \cite{Suz2}).
Applying the argument in the proof of Theorem 4.22 in \cite{Suz}
using $\gamma$ instead of the action used there, we obtain the following consequence.
(By finite generatedness, in this case the proof becomes easier than the one there.)
For any non-empty set $\mathcal{Q}$ of prime numbers,
there is an amenable minimal free purely infinite dynamical system $\alpha_\mathcal{Q}$ of $\mathbb{F}_d$ on the Cantor set whose $K_0$-group $G$ satisfies the following condition.
$$\{p\in \mathcal{P}: [1]_0\in pG\}= \mathcal{Q}.$$
Here $\mathcal{P}$ denotes the set of all prime numbers.
The similar statement for $\mathbb{F}_\infty$ is shown in the proof of Theorem 4.22 of \cite{Suz}.
We also denote by $\alpha_\mathcal{Q}$ a dynamical system of $\mathbb{F}_\infty$ satisfying the above conditions.

Now let $M$ be as in the statement.
Put $$\mathcal{R}:=\{p\in \mathcal{P}: K^1(M){\rm \ contains\ an\ element \ of\ order\ }p\}.$$
Then by \cite{Ch0}, $\mathcal{R}$ is finite.
(Indeed, in either case, $M\times [0, 1]^\mathbb{N}$ is a compact Hilbert cube manifold.
Now the main theorem of \cite{Ch0} shows that
$K^1(M)$ is in fact finitely generated.)

Let $\mathcal{G}$ denote the path-connected component of $\mathcal{H}(M)$ containing the identity.
For each non-empty subset $\mathcal{Q}$ of $\mathcal{P}\setminus \mathcal{R}$,
we apply Proposition \ref{Prop:filling} to $\alpha_\mathcal{Q}$ to choose $\beta$ from $\overline{\mathcal{S}}_{\mathcal{G}}(\alpha_\mathcal{Q})$ whose crossed product is a Kirchberg algebra.
For $i=0, 1$, denote by $G_i$ and $H_i$ the $K_i$-group of the crossed products of $\alpha_\mathcal{Q}$ and $\beta$ respectively.
We claim that
$$\tilde{\mathcal{Q}}:=\{p\in \mathcal{P}\setminus \mathcal{R}: [1]_0\in pH_0\}= \mathcal{Q}.$$
Since the cardinal of the power set of $\mathcal{P}\setminus \mathcal{R}$ is continuum,
this ends the proof.
The inclusion $\mathcal{Q}\subset \tilde{\mathcal{Q}}$ is obvious.
To see the converse, let $p\in \tilde{\mathcal{Q}}$ and
take $h\in H_0$ with $ph=[1]_0$.
Denote by $\partial_i$ the third map of the short exact sequence in Proposition \ref{Prop:free}.
Then since $\partial_0([1]_0)=0$,
we have
$p\partial_0(h)=0$.
On the other hand, by the definition of $\mathcal{R}$ and the fact that $G_1$ is torsion free,
there is no element of order $p$ in the third term of the short exact sequence.
Thus $p\partial_0(h)=0$ implies $\partial_0(h)=0$.
Hence there is an element $y$ in the first term of the short exact sequence
with $\sigma_0(y)=h$.
Here $\sigma_i$ denotes the second map in the short exact sequence.
Then from the injectivity of $\sigma_0$
and the equality $ph=[1]_0$,
we must have $py=[1]_0\otimes [1_M]_0$.
Now let $\tau\colon K^0(M) \rightarrow \mathbb{Z}$
be the homomorphism induced from a character on $C(M)$.
Put $w:=(\id \otimes \tau)(y)\in G_0$.
(We identify $G_0$ with $G_0\otimes \mathbb{Z}$ in the obvious way.)
Then we have $pw=(\id\otimes\tau)([1]_0\otimes [1_M]_0)=[1]_0$.
Thus we get $p\in \mathcal{Q}$ as desired.

The proof for general case is done by taking the induced dynamical systems of the actions obtained in above.
For the detail, see the proof of Theorem 4.7 in \cite{Suz} for instance.
\end{proof}
\subsection*{Acknowledgement}
The author was supported by Research Fellow
of the JSPS (No.25-7810) and the Program of Leading Graduate Schools, MEXT, Japan.

\end{document}